\documentclass[10pt]{article}
\font\smallit=cmti10
\font\smalltt=cmtt10

\usepackage{amssymb,latexsym,amsmath,epsfig,amsthm} 
\usepackage{empheq}
\usepackage{comment}
\usepackage{here}
\usepackage{url}
\usepackage{ascmac}
\makeatletter

\renewcommand\section{\@startsection {section}{1}{\z@}
{-30pt \@plus -1ex \@minus -.2ex}
{2.3ex \@plus.2ex}
{\normalfont\normalsize\bfseries\boldmath}}

\renewcommand\subsection{\@startsection{subsection}{2}{\z@}
{-3.25ex\@plus -1ex \@minus -.2ex}
{1.5ex \@plus .2ex}
{\normalfont\normalsize\bfseries\boldmath}}

\renewcommand{\@seccntformat}[1]{\csname the#1\endcsname. }

\makeatother
\newtheorem{theorem}{Theorem}
\newtheorem{lemma}{Lemma}

\theoremstyle{definition}
\newtheorem{definition}{Definition}[section]

\begin{document}

\begin{center}
\uppercase{\bf Two Dimensional Silver Dollar Game}
\vskip 20pt
{\bf Ryohei Miyadera }\\
{\smallit Keimei Gakuin Junior and High School, Kobe City, Japan}. \\
{\tt runnerskg@gmail.com}
\vskip 20pt
{\bf Enchong  Li  }\\
{\smallit Keimei Gakuin Junior and High School, Kobe City, Japan}. \\
{\tt   }
\vskip 10pt
{\bf Akito Tsujii}. \\
{\smallit Keimei Gakuin Junior and High School, Kobe City, Japan}. \\
{\tt   }

\vskip 10pt

\end{center}
\vskip 20pt
\vskip 30pt

\pagestyle{myheadings} 
\markright{\smalltt \hfill} 
\thispagestyle{empty} 
\baselineskip=12.875pt 
\vskip 30pt

\centerline{\bf Abstract}
\noindent
We define a variant of the two-dimensional Silver Dollar game. Two coins are placed on a chessboard of unbounded size, and two players take turns choosing one of the coins and moving it. Coins are to be moved to the left or upward vertically as far as desired.
If a coin is dropped off the board, players cannot use this coin. Jumping a coin over another coin or on another coin is illegal. 
We add another operation;that is to move a coin and push another coin. 
For non-negative integers $w,x,y,z$, by $(w,x,y,z)$ we denote the positions of the two coins, where $(w,x)$ is the position of one coin and $(y,z)$ is the position of the other coin.
the set of $\mathcal{P}$-positions of this game is 
$\{(w,x,y,z):(w-1)\oplus (x-1) \oplus (y-1) \oplus (z-1)\}$.
Next, we make another game by omitting the rule of pushing another coin and 
permitting a jump over another coin. Then, there exist two set $A$ and $B$ such that 
the set of $\mathcal{P}$-positions of this game is 
$\{(w,x,y,z):(w-1)\oplus (x-1) \oplus (y-1) \oplus (z-1)\}$ $\cup A$ $-B$.

\pagestyle{myheadings}
\markright{\smalltt \hfill}
\baselineskip=12.875pt
\vskip 30pt
\section{Introduction and Combinatorial Game Theory Definitions}\label{introductionsection}
Let $\mathbb{Z}_{\geq0}$ and $\mathbb{N}$ be the sets of nonnegative and natural numbers, respectively.

We briefly review some of the necessary concepts of combinatorial game theory by referring to $\cite{lesson}$. 
We study games that are impartial games without drawings; only two outcome classes are possible.
\begin{definition}\label{NPpositions}
$(a)$ A position is referred to as a $\mathcal{P}$-\textit{position} if it is a winning position for the previous player (the player who has just moved), as long as they play correctly at each stage.\\
$(b)$ A position is referred to as an $\mathcal{N}$-\textit{position} if it is the winning position for the next player, as long as they play correctly at each stage.\\
$(c)$ The \textit{disjunctive sum} of the two games, denoted by $\mathbf{G}+\mathbf{H}$, is a super game in which a player may move either in $\mathbf{G}$ or $\mathbf{H}$ but not in both.\\
$(d)$ For any position $\mathbf{p}$ in game $\mathbf{G}$, a set of positions can be reached by one move in $\mathbf{G}$, which we denote as \textit{move}$(\mathbf{p})$. \\
$(e)$  The \textit{minimum excluded value} ($\textit{mex}$) of a set $S$ of nonnegative integers is the least nonnegative integer that is not in S. \\
$(f)$ Let $\mathbf{p}$ be the position in the impartial game. The associated \textit{Grundy number} is denoted by $G(\mathbf{p})$ and is 
 recursively defined by 
	$G(\mathbf{p}) = \textit{mex}\{G(\mathbf{h}): \mathbf{h} \in move(\mathbf{p})\}.$
\end{definition}

\begin{theorem}\label{theoremofsumg}
Let $\mathbf{G}$ and $\mathbf{H}$ be impartial rulesets, and $G_{\mathbf{G}}$ and $G_{\mathbf{H}}$ be the Grundy numbers of game $\mathbf{g}$ played under the rules of $\mathbf{G}$ and game $\mathbf{h}$ played under those of $\mathbf{H}$. Thus, we obtain the following:\\
	$(i)$ For any position $\mathbf{g}$ in $\mathbf{G}$, 
	$G_{\mathbf{G}}(\mathbf{g})=0$ if and only if $\mathbf{g}$ is the $\mathcal{P}$-position.\\
	$(ii)$ The Grundy number of positions $\{\mathbf{g},\mathbf{h}\}$ in game $\mathbf{G}+\mathbf{H}$ is
	$G_{\mathbf{G}}(\mathbf{g})\oplus G_{\mathbf{H}}(\mathbf{h})$.
\end{theorem}

\subsection{Variants of Silver Dollar Game with a Push}
\begin{definition}\label{twocoinwithpush}
We define a variant of the two-dimensional Silver Dollar game. Two coins are placed on a chessboard of unbounded size, and two players take turns choosing one of the coins and moving it. Coins are to be moved to the left or upward vertically as far as desired. If a coin is dropped off the board, players cannot use this coin. Jumping a coin over another coin is illegal, but players can move a coin and push another coin as in Figures \ref{twodropv} and \ref{twodroph}.\\
\end{definition}
By $(w,x,y,z)$ we denote the positions of the two coins, where $(w,x)$ is the position of one coin and $(y,z)$ is the position of the other coin for 
any $w,x,y,z \in Z_{\geq 0}$.

\begin{figure}[H]
\begin{tabular}{ccc}
\begin{minipage}[t]{0.33\textwidth}
\begin{center}
\includegraphics[height=2.5cm]{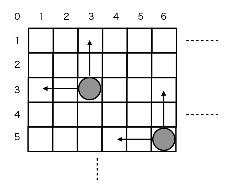}
\caption{}
\label{twodrop}
\end{center}
\end{minipage}
\begin{minipage}[t]{0.33\textwidth}
\begin{center}
 \includegraphics[height=2.5cm]{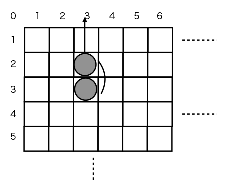}
\caption{}
\label{twodropv}
\end{center}
\end{minipage}
\begin{minipage}[t]{0.33\textwidth}
\begin{center}
\includegraphics[height=2.5cm]{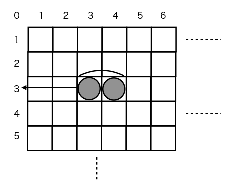}
\caption{}
\label{twodroph}
\end{center}
\end{minipage}
\end{tabular}
\end{figure}

\begin{definition}\label{sumoftwocoin}
We define another game. Two coins are placed on a chessboard of unbounded size, and two players take turns choosing one of the coins and moving it. Coins are to be moved to the left or upward vertically as far as desired. If a coin is dropped off the board, players cannot use this coin. Jumping on or over another coin is permitted.
\end{definition}

\begin{theorem}
In the game of Definition \ref{twocoinwithpush}, 
the set of $\mathcal{P}$-position is $\{(w,x,y,z):(w-1)\oplus (x-1) \oplus (y-1) \oplus (z-1) = 0\}$.
\end{theorem}
\begin{proof}
The proof is almost the same as that for traditional four-pile nim, and we omit it.
\end{proof}

\subsection{Two-Dimensional Silver Dollar Game without a Push}

\begin{definition}\label{twocoinnopush}
We define another game by omitting the rule of pushing two coins together and permitting the jump over another coin
in the game of Definition \ref{twocoinwithpush}.
\end{definition}

\begin{definition}\label{withoutpush}
$(a)$ Let $P_0=\{(w,x,y,z):(w-1)\oplus (x-1) \oplus (y-1) \oplus (z-1) = 0\}.$\\
$(b)$ Let $P_1= \{(w,2n-1,w,2n):w,n \in \mathbb{N}\}$ $\cup \{(w,2n,w,2n-1):w,n \in \mathbb{N}\}$
$\cup \{(2n-1,x,2n,x):x,n \in \mathbb{N}\}$ $\cup \{(2n,x,2n-1,x):x,n \in \mathbb{N}\}$.\\
$(c)$ Let $N_0=\{(2m,2n,2m-1,2n-1):m,n \in \mathbb{N}\}\cup \{(2m,2n-1,2m-1,2n):m,n \in \mathbb{N}\}\cup \{(2m-1,2n-1,2m,2n):m,n \in \mathbb{N}\}\cup \{(2m-1,2n,2m,2n-1):m,n \in \mathbb{N}\}$.
\end{definition}
Note that $N_0 \subset \{(w,x,y,z):(w-1)\oplus (x-1) \oplus (y-1) \oplus (z-1) = 0\}$ and $P_1 \subset \{(w,x,y,z):(w-1)\oplus (x-1) \oplus (y-1) \oplus (z-1) = 1\}$.

\begin{theorem}\label{gamewithoutpush}
The set of $\mathcal{P}$-position for the game of Definition \ref{withoutpush} is 
$P_0 \cup P_1-N_0.$
\end{theorem}

We prove Theorem \ref{gamewithoutpush} by a Theorem in \cite{journalofinf}. As you learn in Subsection \ref{tworooks},
the game of Definition \ref{twocoinnopush} is mathematically the same as the game of Definition \ref{tworookjump}.
Note that one of the authors of the present article is one of the authors of \cite{journalofinf}.

\subsection{Corner the Two Rooks}\label{tworooks}
Here, we introduce the impartial game of ``Corner the Two Rooks," a variant of the classical game of ``Corner the Queen" which was studied in \cite{wythoffpaper}. 	
Instead of the queen used in Wythoff’s game of Nim, we use the two rooks of chess. This game can be considered as a two-dimensional Maya game.

Let us break with chess traditions here and denote fields on the chessboard by pairs of numbers. The field in the upper left corner is denoted by $(0,0)$ and the others are denoted according to a Cartesian scheme: field $(x, y)$ denotes $x$ fields to the right followed by $y$ fields downward.
\begin{definition}\label{tworookjump}   
$(i)$ We define ``Corner the Two Rooks." Two rooks are placed on a chessboard of unbounded size, and two players take turns choosing one of the rooks and moving it. Rooks are to be moved to the left or upward vertically as far as desired. A rook may jump over another rook but not onto another. The first player who cannot make a valid move loses. \\
$(ii)$ By $(x,y,z,w)$ we denote the positions of the two rooks, where $(x,y)$ is the position of one rook and $(z,w)$ is the position of the other rook for 
any $x,y,z,w \in Z_{\geq 0}$.\\
$(iii)$ Let $\mathcal{E} = \{(0,0,1,0), (0,0,0,1), (1,0,0,0), (0,1,0,0)\}$, which consists of terminal positions from which a rook cannot move to another position. See Figures \ref{p1examplep} and \ref{p1examplep2}.
\end{definition}

\begin{figure}[H]
\begin{minipage}[!htb]{0.48\columnwidth}
\includegraphics[width=0.8\columnwidth]{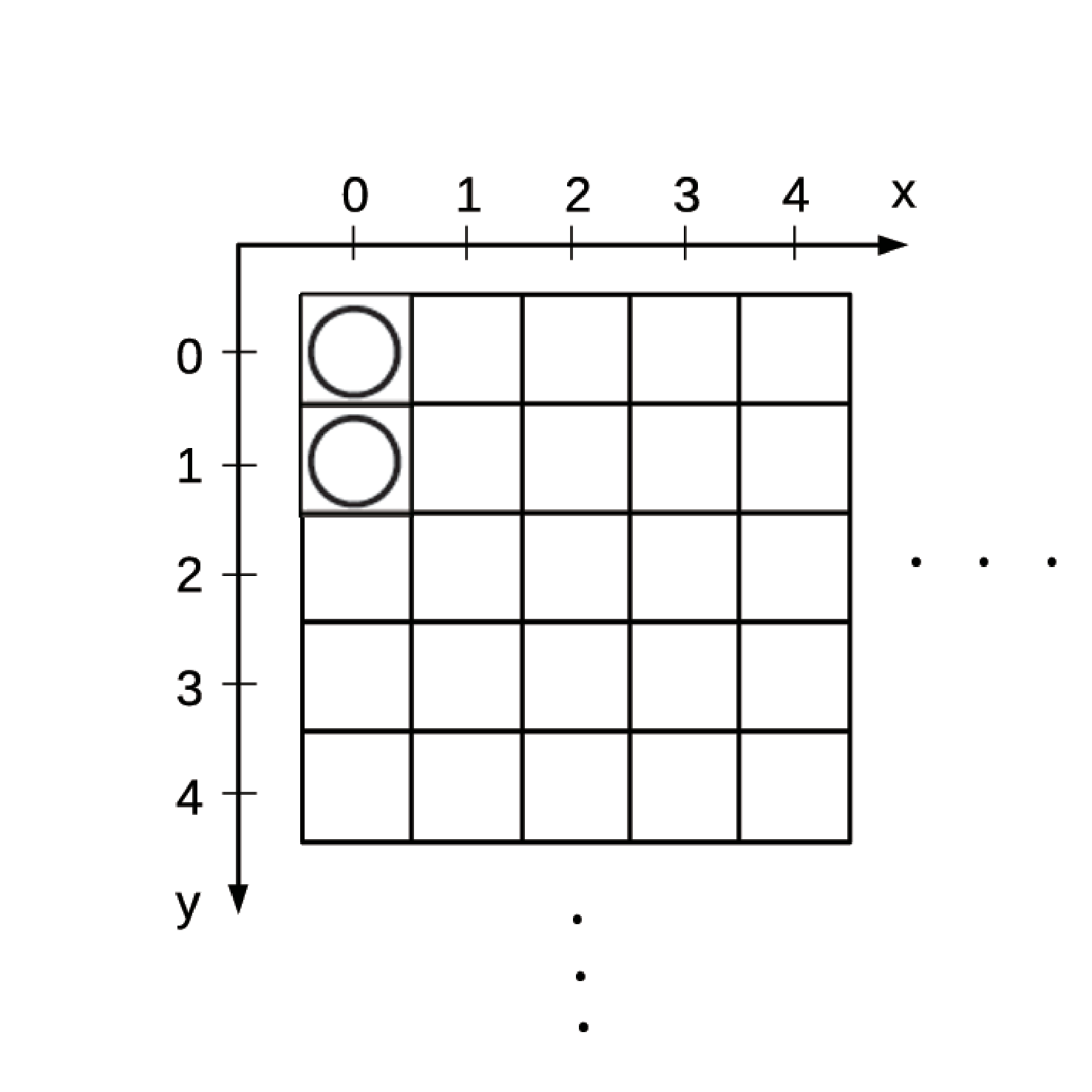}
\caption{Terminal Positions}
\label{p1examplep}	
\end{minipage}
\begin{minipage}[!htb]{0.48\columnwidth}
\includegraphics[width=0.8\columnwidth]{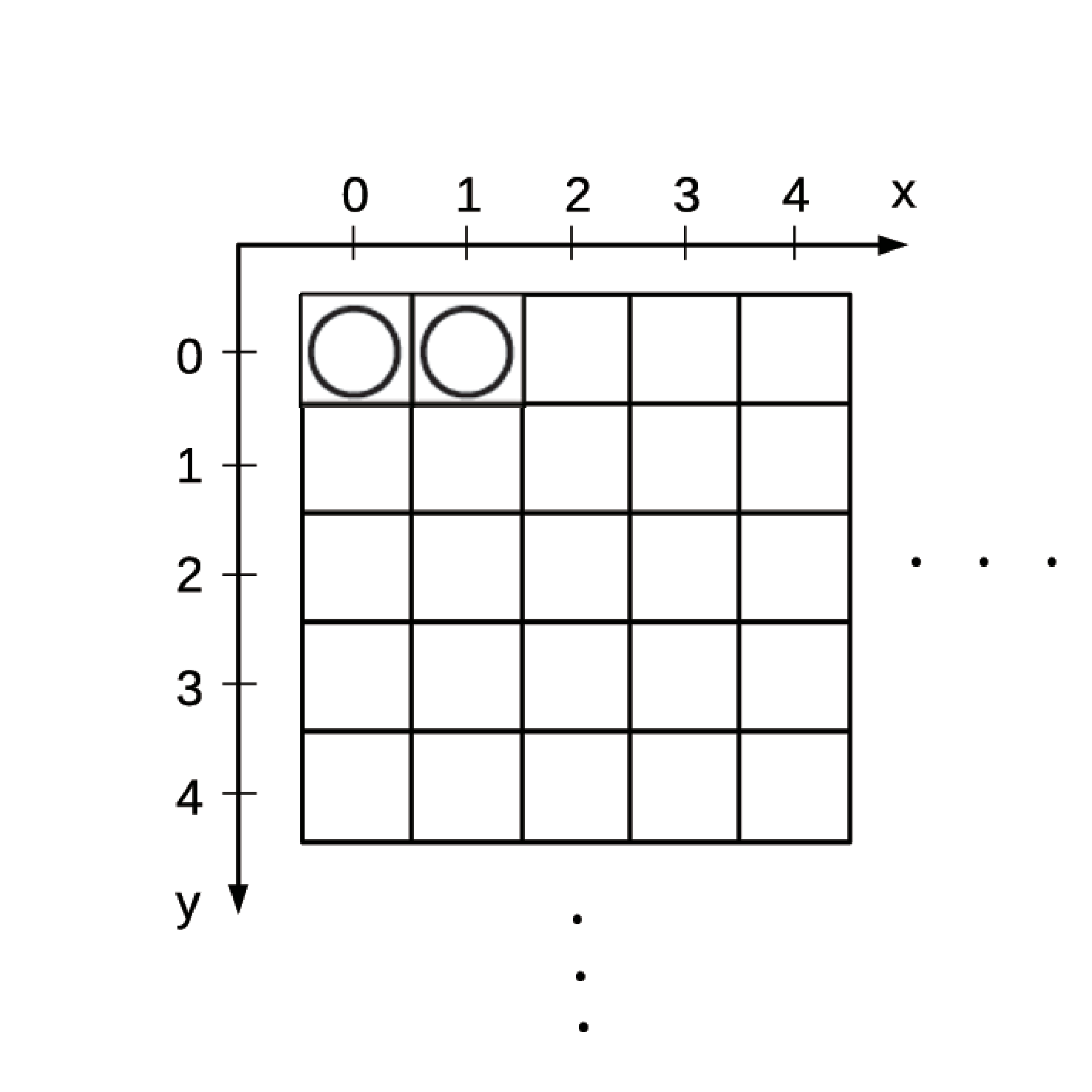}
\caption{Terminal Positions}
\label{p1examplep2}	
\end{minipage}
\end{figure}

\begin{definition}\label{ppositionset}
	For $x,y,z \in Z_{\ge 0}$, we let
\begin{flalign}
	\mathcal{P}_1 & =\{(2n,m,2n+1,m):n,m \in Z_{\geq 0} \}  \cup \{(2n+1,m,2n,m):n,m \in Z_{\geq 0} \}& \nonumber \\
	& \cup \{(n,2m+1, n,2m):n,m \in Z_{\geq 0} \}  \cup \{(n,2m,n,2m+1):n,m \in Z_{\geq 0} \},&  \nonumber  \\
\mathcal{N}_0 & = \{(2n,2m,2n+1,2m+1):n,m \in Z_{\geq 0} \nonumber \\
&  \cup \{(2n+1,2m+1,2n,2m):n,m \in Z_{\geq 0} \} \nonumber \\
& \cup \{(2n+1,2m, 2n,2m+1):n,m \in Z_{\geq 0} \} \nonumber \\
& \cup \{(2n,2m+1,2n+1,2m):n,m \in Z_{\geq 0} \},& \nonumber  \\
	\mathcal{P} &= (\{(x,y,z,w):x \oplus y \oplus z \oplus w =0, \text{ and } x,y,z,w \in Z_{\geq 0}\} \cup \mathcal{P}_1)-\mathcal{N}_0,&  \nonumber  \\
	\mathcal{N} & = \{(x,y,z,w):x \oplus y \oplus z \oplus w \neq 0  \text{ and } x,y,z,w \in Z_{\geq 0}\} \cup \mathcal{N}_0 - \mathcal{P}_1. &  \nonumber 
\end{flalign}
\end{definition}
\begin{figure}[H]
\begin{minipage}[!htb]{0.48\columnwidth}
\includegraphics[width=0.8\columnwidth]{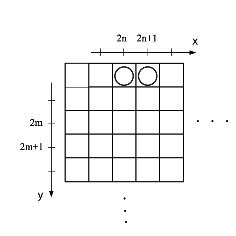}
\caption{An example of elements of $\mathcal{P}_1$}
\label{p1examplep3}	
\end{minipage}
\begin{minipage}[!htb]{0.48\columnwidth}
\includegraphics[width=0.8\columnwidth]{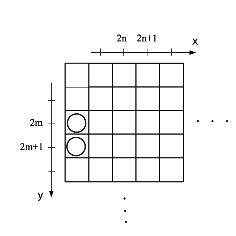}
\caption{An example of elements of $\mathcal{P}_1$}
\label{p1examplep4}	
\end{minipage}
\end{figure}

\begin{figure}[H]
\begin{minipage}[!htb]{0.48\columnwidth}
\includegraphics[width=0.8\columnwidth]{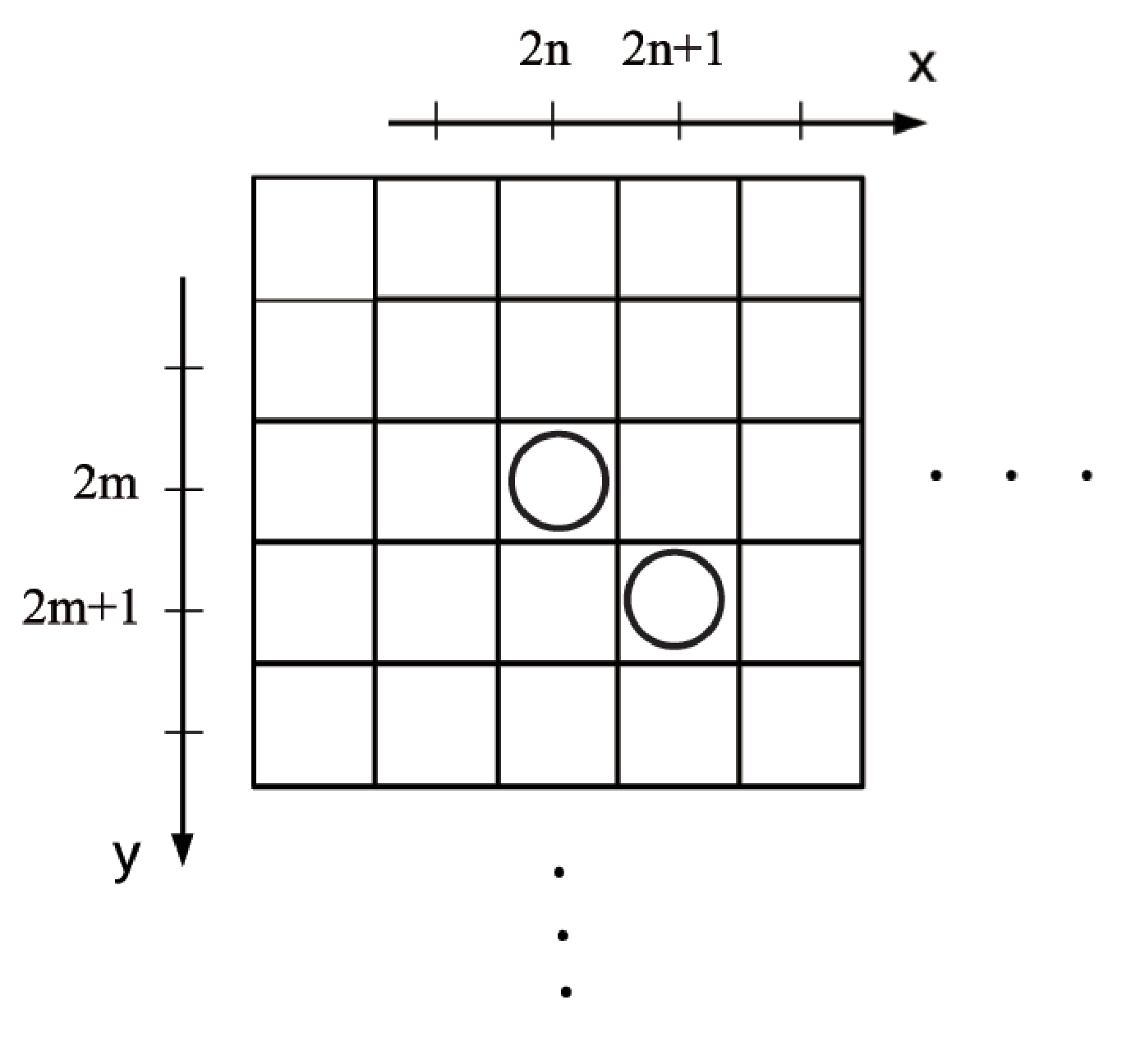}
\caption{An example of elements of $\mathcal{N}_0$}	
\label{n0examplep1}
\end{minipage}
\begin{minipage}[!htb]{0.48\columnwidth}
\includegraphics[width=0.8\columnwidth]{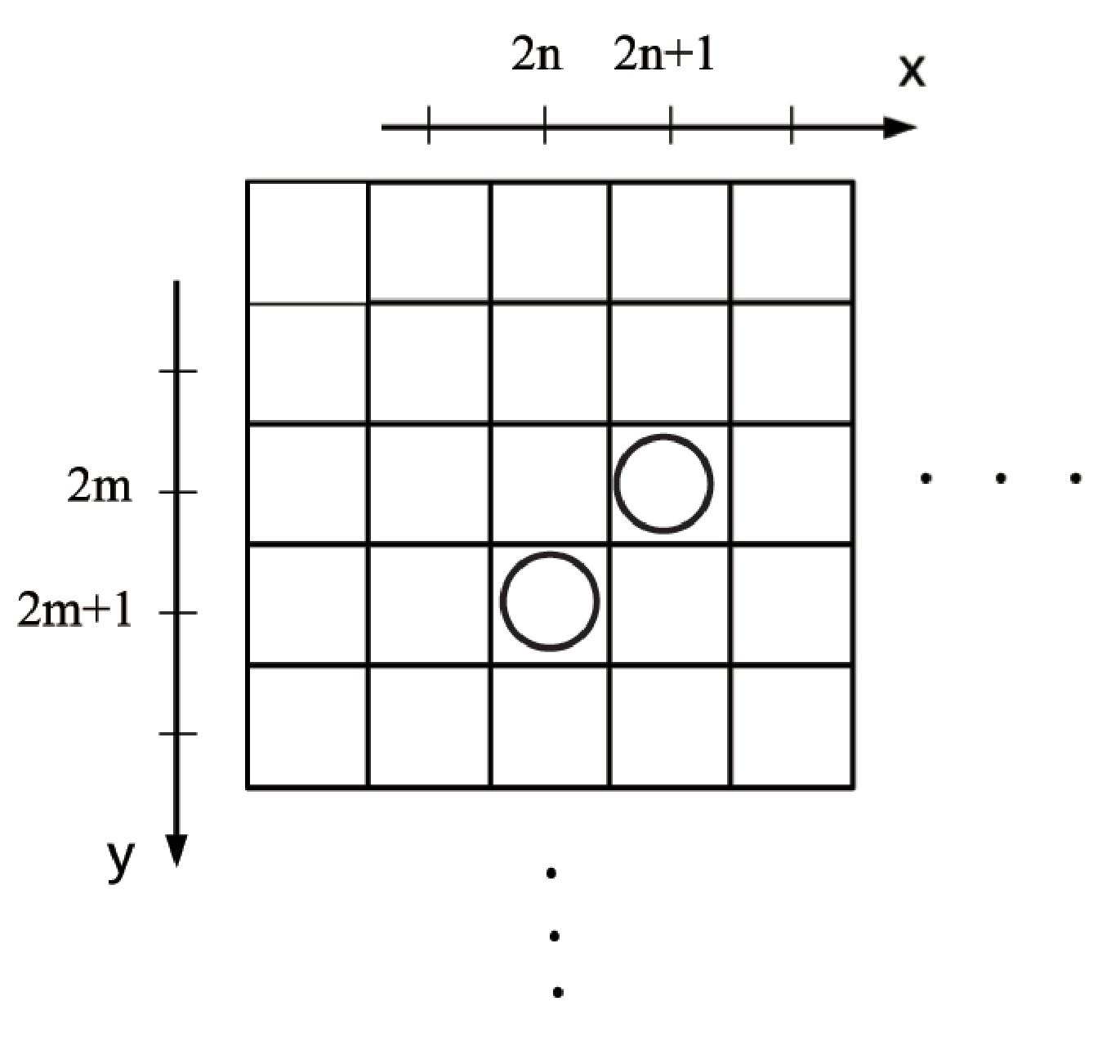}
\caption{An example of elements of $\mathcal{N}_0$}
\label{n0examplep2}	
\end{minipage}
\end{figure}

\begin{theorem}\label{ppositionofrook}
The set of $\mathcal{P}$-positions for the game of Definition \ref{tworookjump} is
$\mathcal{P}$ in Definition \ref{ppositionset}.
\end{theorem}

\begin{lemma}\label{samemath}
The game of Definition \ref{twocoinnopush} is mathematically the same as the game of Definition \ref{tworookjump}.
\end{lemma}
\begin{proof}
In the game of  Definition \ref{twocoinnopush}, moving a coin outside of the chessboard while another is on the chessboard inevitably leads to losing the game.
Therefore, the set of $\mathcal{P}$-positions remains the same even if going out of the chessboard is illegal.
Therefore, two games are mathematically the same.
\end{proof}

By Theorem \ref{ppositionofrook} and Lemma \ref{samemath}, we prove Theorem \ref{gamewithoutpush}.

\end{document}